\documentclass[11pt]{amsart}
\usepackage{float}
\usepackage{graphicx}
\usepackage{amssymb}
\theoremstyle{plain}
\newtheorem{thm}{Theorem}[section]
\newtheorem{lemma}[thm]{Lemma}

\newtheorem{prop}[thm]{Proposition}

\theoremstyle{definition}

\newtheorem{example}[thm]{Example}

\def\card{\mathop{\hbox {card}}\nolimits}

\newcommand{\frg}{\mathfrak{g}}

\newcommand{\frk}{\mathfrak{k}}

\newcommand{\frp}{\mathfrak{p}}

\newcommand{\frt}{\mathfrak{t}}

\newcommand{\frsp}{\mathfrak{sp}}

\newcommand{\bbH}{\mathbb{H}}

\newcommand{\bbR}{\mathbb{R}}

\newcommand{\bbZ}{\mathbb{Z}}

\begin{document}

\title[Distribution of spin norm along pencils]
{Distribution of spin norm along pencils: the $Sp(p, q)$ case}

\author{Chao-Ping Dong}
\address[Dong]{School of Mathematical Sciences, Soochow University, Suzhou 215006,
P.~R.~China}
\email{chaopindong@163.com}

\author{Zhan Ying}
\address[Ying]{School of Mathematical Sciences, Soochow University, Suzhou 215006,
P.~R.~China}
\email{1171049153@qq.com}

\abstract{As a sequel to \cite{D16} and Theorem C of \cite{D20}, this paper shows that for $Sp(p,q)$, the spin norm strictly increases along any Vogan pencil once it goes beyond the unitarily small convex hull.}

\endabstract



\maketitle


\section{Introduction}
Throughout this paper, we fix $G$ as $Sp(p, q)$, which is \emph{not} Hermitian symmetric and has a maximal compact subgroup $K:=Sp(p)\times  Sp(q)$.  We assume that $p\leq q$ and put $n:=p+q$. Let $\frg$ (resp., $\frk$) be the complexified Lie algebra of $G$ (resp., $K$). Let $\frg=\frk+\frp$ be the Cartan decomposition. By a $K$-type, we mean an irreducible representation of $K$.
We may and we will identify a $K$-type with its highest weight. For example, $\frp$ is a $K$-type with highest weight $\beta$. Let $T$ be a maximal torus of $K$. Fix
\begin{equation}\label{k-posroots}
\Delta^+(\frk, \frt):=\{e_i\pm e_j\mid 1\leq i<j\leq p \mbox{ or } p+1\leq i<j\leq n\}\cup
\{2e_i\mid 1\leq i\leq n\}.
\end{equation}

Let $\pi$ be an infinite-dimensional irreducible representation of $G$.
A result of Vogan in \cite{V80} says that the $K$-types of $\pi$ is a union of \emph{pencils}:
\begin{equation}\label{pencil}
P(\mu):=\{\mu+ m\beta \mid m\in \bbZ_{\geq 0}\}.
\end{equation}
Here $\mu$ is the highest weight of a $K$-type. The \emph{unitarily small convex hull} was introduced by Salamanca-Riba and Vogan in \cite{SV}. It will be recalled in Section \ref{sec-pre}. We say $\mu$ is \emph{u-small} if it lies within the unitarily small convex hull; otherwise, we say $\mu$ is \emph{u-large}.

\begin{thm}\label{thm-reduction-pencil-sppq}
For $Sp(p, q)$, the
spin norm increases strictly along any pencil once it goes beyond
the u-small convex hull. Namely,  for any u-large weight $\mu$ such
that $\mu-\beta$ is dominant, we have
\begin{equation}\label{pencil-reduction-step-complex}
\| \mu \|_{\mathrm{spin}} > \|\mu-\beta \|_{\mathrm{spin}}.
\end{equation}
\end{thm}

The notion spin norm was introduced in \cite{D13}. It will be recalled in Section \ref{sec-pre}. Let us explain motivation of the above theorem. Now assume that $\pi$ is \emph{unitary}. An ongoing research topic is the \emph{Dirac cohomology} $H_D(\pi)$. See \cite{V97,HP} and references therein for necessary background on Dirac cohomology. To compute $H_D(\pi)$, Theorem \ref{thm-reduction-pencil-sppq} helps one to reduce $K$-types along any pencil $P(\mu)$ of $\pi$. Indeed, if $\mu$ is u-large, it suffices to consider $\mu$ itself; otherwise, it suffices to consider the finitely many u-small $K$-types among $P(\mu)$.

The analogue of  Theorem \ref{thm-reduction-pencil-sppq} has been obtained for complex Lie groups in \cite{D16} and for $SL(n, \bbR)$, $SL(n, \bbH)$ and all non-Hermitian symmetric exceptional simple Lie groups in Theorem C of \cite{D20}. Since 2018, the first named author has kept trying to attain Theorem \ref{thm-reduction-pencil-sppq} but in vain. We realized that it was more difficult in the $Sp(p, q)$ case since the unitarily small convex hull became much more complicated. From a conventional point of view, this is a problem of linear programming or integer partition, and it is related to the arrangement of Young diagrams. Therefore, computationally, it is easy to verify the theorem for relatively small values of $p$ and $q$. However, as $p$ and $q$ increase, the constraints increase rapidly to the extent that even computers cannot handle them. The spin norm is defined via a minimum. For our problem, one might generally consider how to attain this minimum. Nevertheless, in the discrete setting we are dealing with, attaining the minimum is extremely difficult. In this paper, we present a completely different approach that allows us to circumvent this difficulty, which may serve as a new technique for solving similar problems. We hope the current work could shed some light on dealing with the unitarily small convex hull.

The paper is organized as follows: Section \ref{sec-pre} recalls necessary preliminaries. Section \ref{sec-proof} proves Theorem \ref{thm-reduction-pencil-sppq}.

\section{Preliminaries}\label{sec-pre}
Let $G=Sp(p, q)$ and continue with the notation in the introduction.

\subsection{Root data on $\frsp(p,q)$}
Recall that a $\Delta^+(\frk, \frt)$ has been fixed in \eqref{k-posroots}. Denote  $\alpha_j=e_j-e_{j+1}(j\neq p,n)$, $\alpha_p=2e_p$, and $\alpha_{n}=2e_{n}$. They are the simple roots of $\Delta^+(\frk,\frt)$. The corresponding fundamental dominant weights are $\omega_j=\sum_{i=1}^j e_i$ for $1\le j\le p$ and $\omega_j=\sum_{i=p+1}^j e_i$ for $p+1\le j\le n$.

For $Sp(p,q)$, we choose the Vogan diagram as follows:
\begin{figure}[H]
\centering
\scalebox{1.1}{\includegraphics{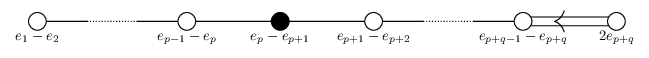}}
\caption{Vogan diagram for $Sp(p, q)$}
\end{figure}
In this case,
\begin{gather*}
	\Delta^+(\frp,\frt)=\{e_i\pm e_j\mid 1\le i\le p, \, p+1\le j\le n\}.
\end{gather*}
Let $\gamma_i=e_i-e_{i+1}(1\le i\le n-1)$ and $\gamma_{n}=2e_{n}$. They are the simple roots of $\Delta^+(\frg,\frt)$. The corresponding fundamental dominant weights are $\xi_j=\sum_{i=1}^je_i(1\le j\le n)$. We denote the half sum of the roots in $\Delta^+(\frg,\frt)$,\ $\Delta^+(\frk,\frt)$ and $\Delta^+(\frp,\frt)$ by $\rho$,\ $\rho_c$ and $\rho_n$, respectively. Clearly,
\begin{align*}
	\rho&=(n,n-1,\dots,q+1\mid q,q-1,\dots,1),\quad &&\rho_c=(p,p-1,\dots,1\mid q,q-1,\dots,1),\\
	\rho_n&=(q,q,\dots,q\mid 0,0,\dots,0),\quad &&\beta=(1,0,\dots,0\mid 1,0,\dots,0).
\end{align*}
\subsection{spin norm}
Notice that $\frk$ has no center, we have
\[\Delta^+(\frg,\frt)=\Delta^+(\frk,\frt)\cup\Delta^+(\frp,\frt).\]
Let $C$ is the dominant Weyl chamber corresponding to $\Delta^+(\frk,\frt)$. Denote by $C_\frg$ the closed Weyl chamber corresponding to $\Delta^+(\frg,\frt)$. Then $C_\frg$ is contained in $C$. Define
\[W(\frg,\frt)^1=\{w\in W(\frg,\frt)\mid w(C_\frg)\subset C\}.\]
As in \cite{D20},
\[\{w\Delta^+(\frg,\frt)\mid w\in W(\frg,\frt)^1\}\]
are exactly all the choices of positive roots systems for $\Delta(\frg,\frt)$ containing $\Delta^+(\frk,\frt)$. Let us enumerate the elements of $W(\frg,\frt)^1$ as $w^{(0)}=e,\dots,w^{(s-1)}$; the enumeration will specified in Section \ref{subsec-rhon}. For $0\le \ell\le s-1$, put
\[(\Delta^+)^{(\ell)}(\frp,\frt)=w^{(\ell)}\Delta^+(\frp,\frt),\quad (\Delta^+)^{(\ell)}(\frg,\frt)=\Delta^+(\frk,\frt)\cup w^{(\ell)}\Delta^+(\frp,\frt).\]
Denote by $\rho_n^{(\ell)}$ the half sum of the positive roots in $(\Delta^+)^{(\ell)}(\frp,\frt)$. Then we have
\[w^{(\ell)}\rho=\rho_c+\rho_n^{(\ell)},\quad 0\le \ell\le s-1.\]
Now as in \cite{D13}, the spin norm of a $\Delta^+(\frk,\frt)$-dominat weight $\nu$ is defined to be
\[\Vert \nu \Vert_{\mathrm{spin}}:=\min_{0\le\ell\le s-1}\Vert \{\nu-\rho_n^{(\ell)}\}+\rho_c\Vert.\]

The Weyl group $W(\frg,\frt)$ acts by permutations and sign changes, and that a weight is $\Delta^+(\frk,\frt)$-dominat if and only if both its first $p$ coordinates and its last $q$ coordinates form a weakly decreasing sequence. For $i<j$, let
$$
s_{i,j}=s_{e_i-e_{i+1}} s_{e_{i+1}-e_{i+2}}\cdots s_{e_{j-1}-e_{j}}.
$$
We make the convention that $s_{i,i}=e$. Then
\[W(\frg,\frt)^1=\left\{s_{p,i_p} s_{p-1,i_{p-1}}\cdots s_{2,i_2} s_{1,i_1}\mid 1\le i_1<i_2<\cdots<i_p\le n\right\},\]
In particular, $i_{k-1}+1\le i_k\le q+k$, where $i_0:=0$. Thus $W(\frg,\frt)^1$ has cardinality
$s:=\dbinom{n}{p}$.

\subsection{u-small $K$-types}

The \emph{unitarily small convex hull} was defined by Salamanca-Riba and Vogan in \cite{SV} as
$$
R(\Delta(\frp, \frt)):=\left\{\sum_{\alpha\in\Delta(\frp, \frt)}  b_{\alpha} \alpha\mid 0\leq b_{\alpha}\leq 1\right\}.
$$

The following result essentially comes from Theorem 6.7 of \cite{SV}.

\begin{lemma}\emph{(Lemma 4.4 of \cite{D20})}\label{lemma-Dong}
Let $\nu$ be any $\Delta^+(\frk, \frt)$-dominat weight. Then $\nu$ is u-small if and only if $\langle\nu+2\rho_c, w\xi_i\rangle \le 2\langle\rho, \xi_i\rangle$ for all $1\le i\le n$ and all $w\in W(\frg, \frt)^1$.
\end{lemma}

From now on, we always use
\begin{equation}\label{mu}
\mu=(a_1,\dots,a_p\mid b_1,\dots,b_q)
 \end{equation}
to denote a $\Delta^+(\frk,\frt)$-dominant weight consisting of non-negative integers.
By the above lemma, we obtain a more explicit criterion for $\mu$ being u-large.

\begin{prop}\label{prop-u-large}
	A weight $\mu$ in the form of \eqref{mu} is u-large if and only if there exist $0\le f\le p$ and $0\le g\le q$ such that
\begin{equation}\label{eq4}
\sum_{i=1}^f a_i+\sum_{j=1}^g b_j>2pq-2(p-f)(q-g).
\end{equation}
\end{prop}
\begin{proof}
	By Lemma \ref{lemma-Dong}, $\mu$ is u-small if and only if $\langle\mu + 2\rho_c, w\xi_i\rangle \le 2\langle\rho,\xi_i\rangle$ for all $w\in W(\frg, \frt)^1$ and all $1\leq i\leq n$. Taking into account the action of $w$ on weights, the latter condition is equivalent to
	\[\sum_{i=1}^f \big(a_i+(p+1-i)\big)+\sum_{j=1}^g \big(b_j+(q+1-j)\big)\le 2\sum_{t=1}^{f+g}(n+1-t),\]
	where $0\le f\le p$ and $0\le g\le q$. The above inequality is further equivalent to
	\begin{align*}
		\sum_{i=1}^f a_i+\sum_{j=1}^g b_j&\le  2\left(\sum_{t=1}^{f+g}(n+1-t)-\sum_{i=1}^f (p+1-i)-\sum_{j=1}^g (q+1-j)\right)\\
		&=2\left(\sum_{i=1}^f q+\sum_{j=1}^g (p-f)\right)\\
        &=2(fq+gp)-2fg\\
        &=2pq-2(p-f)(q-g).
	\end{align*}
	The lemma follows.
\end{proof}

Let $\mu$ be as in \eqref{mu}.
For $0\le f\le p$ and $0\le g\le q$, put
\begin{equation}\label{mufg}
\mu_{f, g}=(a_1,\dots,a_f, q-g, \dots, q-g \mid b_1,\dots,b_g, p-f, \dots, p-f).
\end{equation}
Let $S(v)$ denote the sum of the coordinates of a real vector $v$. Then

\[S(\mu_{f, g})=\sum_{i=1}^f a_i+\sum_{j=1}^g b_j+2(p-f)(q-g).\]

This shows that $\mu$ is u-small if and only if $S(\mu_{f, g})\le 2pq$ for all $0\le f\le p$ and all $0\le g\le q$. In particular, if $a_1\le q$ and $b_1\le p$, then $\mu$ is u-small. This proves the following lemma.

\begin{lemma}\label{lemma-u-large}
	Assume that $\mu$ is u-large. Then $a_1\ge q+1$ or $b_1\ge p+1$.
\end{lemma}

For convenience, for $0\le f\le p$ and $0\le g\le q$, we denote by $\Lambda_{f, g}$ the set of all weights $\mu$ in the form of \eqref{mu} satisfying
\[\sum_{i=1}^f a_i+\sum_{j=1}^g b_j>2pq-2(p-f)(q-g).\]
Namely,
\begin{equation}\label{sum-mufg}
S(\mu_{f, g})> 2pq.
\end{equation}

\subsection{Domination}
For real vectors $u=(u_1,\dots, u_n)$ and $v=(v_1, \dots, v_n)$, we say that $u$ \emph{dominates} $v$ if $u_i\ge v_i$ for all $1\le i\le n$ and  strict inequality happens for at least one $i$. In this case, we write $u\gg v$. We write $\{u\}$ for the vector obtained by rearranging the absolute values of the coordinates of $u$ in decreasing order. Let $e_i$ denote the $n$-dimensional vector whose entries are all zero except that the $i$-th entry equals $1$.

\begin{lemma}\label{lemma-dominat}
Let $x=(x_1,\dots,x_n)$ be consisting of integer entries.   Then, if $x_s \geq 0$ for some $1\le s\le n$, one has $\{x+e_s\}\gg\{x\}$. If $x_s\leq 0$, then $\{x-e_s\}\gg\{x\}$.
\end{lemma}
\begin{proof}
	For the first assertion, consider the multiset
	\[K=\{x_j\mid |x_j|>|x_s|\},\quad T=\{x_j\mid |x_j|\le |x_s|,j\neq s\}.\]
	Arrange the elements of $K$ and $T$ in decreasing order and write them as $k_1,\dots,k_l$ and $t_1,\dots,t_r$,  respectively. Then
	\[\{x+e_t\}=(k_1,\dots,k_l,a_s+1,t_1,\dots,t_r),\quad \{x\}=(k_1,\dots,k_l,a_s,t_1,\dots,t_r).\]
The first assertion follows. The second assertion can be proven similarly.
\end{proof}
Let $\nu$ denote an $n$-dimensional integer vector, and define the \emph{$\frk$-value} of $\nu$ as follows
\[\Vert \nu\Vert_{\frk}:=\Vert\{\nu\}+\rho_c\Vert.\]
Clearly, for any $\nu_1, \nu_2\in\bbZ^n$ such that $\{v_1\}\gg\{v_2\}$, one has
\[\Vert \nu_1\Vert_\frk>\Vert \nu_2\Vert_\frk.\]

\subsection{The structure of $\rho_n^{(\ell)}$}\label{subsec-rhon}
 Recall for any $0\leq \ell\leq s-1$, we have
 $$
 \rho_{n}^{(\ell)}:=w^{(\ell)}\rho-\rho_c.
 $$
Denote by $\Omega_{p,q}=\{\rho_n^{(\ell)}\mid 0\leq \ell\leq s-1\}$. This set consists of $(x_1, \dots, x_p \mid y_1, \dots, y_q)\in\bbZ^n$ satisfying the following conditions:	
\begin{itemize}
	\item[$\bullet$] $q\ge x_1\ge \cdots\ge x_p\ge 0$,\ $p\ge y_1\ge \cdots\ge y_q\ge 0$;
	\item[$\bullet$] $\sum_{i=1}^px_i+\sum_{j=1}^qy_j=pq$;
	\item[$\bullet$] $y_j=\card\{i\mid q-x_i\ge j\}=\max\{i\mid q-x_i\ge j\}$ for any $1\leq j\leq q$.
\end{itemize}

\begin{lemma}\label{lemma-omegapq}
	Take any $(x_1, \dots, x_p \mid y_1, \dots, y_q)\in\Omega_{p, q}$. Then
	\[\sum_{i=1}^f x_i+\sum_{j=1}^g y_i\le fq+(p-f)g\]
for all $1\le f\le p$, $1\le g\le q$.
\end{lemma}
\begin{proof}
	Recall that $y_j=\card\{i\mid q-x_i\ge j\}$ for $1\leq j\leq q$. Thus
	\[\sum_{j=1}^g y_j=\sum_{i=1}^p\min(q-x_i, g)\]
	since each  $1\leq i\leq p$ contributes $1$ to $y_1,\dots,y_{\min (q-x_i, g)}$. This shows that
	\begin{align*}
    \sum_{i=1}^f x_i+\sum_{j=1}^g y_j&=\sum_{i=1}^f x_i+\sum_{i=1}^p\min(q-x_i, g)\\
      &=\sum_{i=1}^f \big(x_i+\min(q-x_i, g)\big)+\sum_{i=f+1}^p\min(q-x_i, g)\\
      &\le fq+(p-f)g.
      \end{align*}
\end{proof}

We can arrange the members of $\Omega_{p, q}$ in the \emph{lexicographic order}. Namely, $\rho_n^{(j)}$  is ahead of $\rho_n^{(k)}$ if the first nonzero coordinate of $\rho_n^{(j)}-\rho_n^{(k)}$ is positive. In particular,
\[\rho_n^{(0)}=(\underbrace{q,\dots,q}_{p}\mid 0,\dots,0),\quad \rho_n^{(s-1)}=(0,\dots,0\mid \underbrace{p,\dots,p}_{q}).\]
Correspondingly, we enumerate the elements of $W(\frg,\frt)^1$ as $w^{(0)}, w^{(1)}, \dots, w^{(s-1)}$, as previously agreed.
To simplify notation, we always denote $\mu-\rho_n^{(\ell)}$ by $\mu_\ell$.

Unless specified otherwise, let $K_\ell=(k_1, \dots, k_p)$ denote the first $p$ coordinates of $\rho_n^{(\ell)}$. By convention, $k_0=q$ and $k_{p+1}=0$. Note that the last $q$ coordinates of $\rho_n^{(\ell)}$ are completely determined by its first $p$ coordinates. Indeed, for $1\le j\le p+1$, put
$$
R_\ell(j):=(\underbrace{p-j+1,\dots,p-j+1}_{k_{j-1}-k_j})
$$
if $k_{j-1}-k_j\neq 0$ and put $R_\ell(j)$ as the empty array otherwise.

If we denote the last $q$ coordinates by $R_\ell$, then
\[R_\ell=\big(R_\ell(1),\cdots,R_\ell(p+1)\big).\]

\begin{example}\label{example-R}
	Let $p=4$ and $q=6$. If the first $p$ coordinates of $\rho_n^{(\ell)}$ are $(5,2,2,1)$, then $R_\ell(3)$ is an empty array. Therefore,
		\[R_\ell=\big(R_\ell(1)\mid R_\ell(2)\mid R_\ell(4)\mid R_\ell(5)\big)=(4\mid3,3,3\mid1\mid0).\]\hfill\qed
\end{example}

For any $0\leq \ell \leq s-2$, there must exist $1\le j\le p$ such that $k_{j}>k_{j+1}$. By the previous argument, the following vector is a member of $\Omega_{p, q}$:
$$
(K_\ell-e_{j}\mid R_\ell(1), \dots, R_\ell(j), \underbrace{p-j+\textbf{1}, p-j, \dots, p-j}_{k_{j}-k_{j+1}}, R_\ell(j+2), \dots, R_\ell(p+1))
$$
It  depends on both $\ell$ and $j$. Let us denote it by $\rho_n^{(L_{\ell,j})}$.
Thus
\begin{equation}\label{eq5}
\rho_n^{(\ell)}-\rho_n^{(L_{\ell, j})}=(\underbrace{0,\dots,0,1}_{j},0,\dots,0\mid 0,\dots,0,\underbrace{-1,0,\dots,0}_{k_{j}}).\end{equation}

\section{Proof of Theorem \ref{thm-reduction-pencil-sppq}}\label{sec-proof}
Throughout this section, we continue to use the previous notation.
In particular, $\mu$ is of the form \eqref{mu}, $\mu$ is u-large and $\mu-\beta$ is dominant for $\Delta^+(\frk, \frt)$.

\subsection{Technical preparations}
We say that $\rho_n^{(\ell)}$ is \emph{$\mu$-deficient} if $\Vert\mu_\ell\Vert_\frk\le \Vert\mu_\ell-\beta\Vert_\frk$.

\begin{example}\label{example-deficient} Let $p=3$ and $q=5$. Consider two weights
	\[\nu=(6,5,5 \mid 7,6,6,6,6)\quad\mbox{and}\quad \kappa=(2,0,0 \mid 7,6,6,6,6).\]They are u-large beacuse $S(\nu)$ and $S(\kappa)$ greater than $2pq$. Clearly, no $\nu$-deficient $\rho_n^{(\ell)}$ can occur. For
	\[\rho_n^{(35)}=(4,0,0 \mid 3,2,2,2,2),\]
	we compute
\[\Vert\kappa_{35}\Vert_\frk^2=285<287=\Vert\kappa_{35}-\beta\Vert_\frk^2.\]
Thus $\rho_n^{(35)}$ is $\kappa$-deficient but not $\nu$-deficient.\hfill\qed
\end{example}

Assume that there exists a $\mu$-deficient $\rho_n^{(\ell)}=(k_1,\dots,k_p\mid r_1,\dots,r_q)$, and let \[\mu_\ell=(m_1, \dots, m_p \mid n_1, \dots, n_q).\]
Put
\begin{align*}
	M_\ell&=\card\{i\neq 1\mid |m_i|\ge|m_1|\},\quad &&N_\ell=\card\{i\neq 1\mid |n_i|\ge|n_1|\}, \\
	M_\ell^+&=\card\{i\mid |m_i|>|m_1|\},\quad &&N_\ell^+=\card\{i\mid |n_i|>|n_1|\}.
\end{align*}
Since $\rho_n^{(\ell)}$ is assumed to be $\mu$-deficient, it follows form Lemma 2.4 that $m_1\ge 1$ and $n_1\ge 1$ cannot occur simultaneously. Therefore, $m_1\leq 0$ or $n_1\leq 0$.  Thus we must have $k_1\ge a_1$ or $r_1\geq b_1$.

Now there are  three possible cases:
\begin{itemize}
\item[$\bullet$] $m_1\ge 1$ and $n_1\le 0$. In this case,
\begin{equation}\label{eq6}
\Vert\mu_\ell-\beta\Vert_\frk^2-\Vert\mu_\ell\Vert_\frk^2=2\big((|n_1|-|m_1|)+(q-p)+(M_\ell-N_\ell^+)+1\big)\ge 0.
\end{equation}

\item[$\bullet$] $m_1\le 0$ and $n_1\ge 1$. In this case,
\begin{equation}\label{eq7}
\Vert\mu_\ell-\beta\Vert_\frk^2-\Vert\mu_\ell\Vert_\frk^2=2\big((|m_1|-|n_1|)+(p-q)+(N_\ell-M_\ell^+)+1\big)\ge 0.
\end{equation}

\item[$\bullet$] $m_1\le 0$ and $n_1\le 0$.  Then $a_1=m_1+k_1\leq k_1\leq q$ and $b_1=n_1+r_1\leq r_1\leq p$. This contradicts with Lemma \ref{lemma-u-large} since $\mu$ is assumed to be u-large.
\end{itemize}

To finish this subsection, we make some preparations for Proposition \ref{prop-2p+1}. Observe that
\[r_1=\left\{
\begin{array}{lc}
	p, & \mbox{if } k_1<q, \\
	p-j+1, &\mbox{if } k_{j-1}=q \mbox{ and } k_j<q.
\end{array}\right.\]
Hence $q\le k_1+r_1\le p+q-1$.

Since it is assumed that $\mu-\beta$ is $\Delta^+(\frk,\frt)$-dominant, then
\[a_1-1\ge a_2\ge\cdots\ge a_p\quad \mbox{and}\quad b_1-1\ge b_2\ge \cdots\ge b_q.\]
Notice that $k_0=\dots=k_{p-r_1}=q$ and $r_0=\dots=r_{q-k_1}=p$, we find that if $p-r_1>1$, we have $k_1=q$ and
\[m_1>m_2\ge\cdots\ge m_{p-r_1}.\]
If $q-k_1>1$, we have $r_1=p$ and
\[n_1>n_2\ge\cdots\ge n_{q-k_1}.\]
If $q-k_1\le 1$ or $p-r_1\le 1$, then $k_1\ge q-1$ or $r_1\ge p-1$.

 We consider two cases.
\begin{itemize}
	\item[$\bullet$] $b_1\ge 2p$. In this case,
	\begin{equation}\label{ineq1}
    |n_1|>|n_2|\ge\cdots\ge |n_{q-k_1}|\quad\mbox{or}\quad k_1\ge q-1.
    \end{equation}
	Thus we always have $N_\ell\le q-(q-k_1)=k_1$.
	\item[$\bullet$] $a_1\ge 2q$. The same argument shows that $M_\ell\le r_1$.
\end{itemize}

\subsection{The basic case}
If $a_1\ge q+1$ and $b_1\ge p+1$, then $\{\mu_\ell\}\gg\{\mu_\ell-\beta\}$ for all $0\leq \ell\leq s-1$. Therefore,
\[\Vert\mu\Vert_{\mathrm{spin}}=\Vert\mu_{i_0}\Vert_\frk>\Vert \mu_{i_0}-\beta\Vert_\frk\ge \Vert\mu-\beta\Vert_{\mathrm{spin}}.\]
Here $0\leq i_0\leq s-1$ is an index where $\mu$ attains its spin norm.
We say $\mu$ is an \emph{$R$-weight} if $a_1\le q$ and $b_1\ge p+1$; we say $\mu$ is an \emph{$L$-weight} $a_1\ge q+1$ and $b_1\le p$. Now it suffices to consider that $\mu$ is an $R$-weight or an $L$-weight.

We \emph{claim} that if $\mu$ is an $R$-weight and $a_1+b_1\ge 2p+q$,  then no $\mu$-deficient $\rho_n^{(\ell)}$ can occur. Actually, if there exists $\rho_n^{(\ell)}=(k_1,\dots,k_p\mid r_1,\dots,r_q)$ which is $\mu$-deficient, then $\mu_\ell=(m_1,\dots,m_p\mid n_1,\dots,n_q)$. Then \eqref{eq7} must hold, that is,
\[(|m_1|-|n_1|)+(p-q)+(N_\ell-M_\ell^+)+1\ge 0.\]
Because $k_1+r_1\le p+q-1$, the above  inequality implies
\[N_\ell-M_\ell^+\ge (a_1+b_1)-(k_1+r_1)+(q-p)-1\ge (a_1+b_1)-2p.\]
Since $N_\ell\le q-1$, we see that
\[ \quad a_1+b_1\le 2p+N_\ell-M_\ell^+ \le 2p+q-1.\]
This contradicts with the assumption that $a_1+b_1\ge 2p+q$ and proves our claim.
Similarly, if $\mu$ is a $L$-weight and $a_1+b_1\ge 2q+p$, then no $\mu$-deficient $\rho_n^{(\ell)}$ can occur. Since no $\mu$-deficient $\rho_n^{(\ell)}$ exists, one can similarly deduce that $\Vert\mu\Vert_{\mathrm{spin}}>\Vert \mu-\beta\Vert_{\mathrm{spin}}$.

Therefore, it remains to consider the following cases:
\begin{itemize}
	\item[$\bullet$] $\mu$ is an $R$-weight such that $a_1+b_1\le 2p+q-1$;
	\item[$\bullet$] $\mu$ is an $L$-weight such that $a_1+b_1\le p+2q-1$.
\end{itemize}

\begin{prop}\label{prop-2p+1}
	For any u-large $\mu$ in \eqref{mu} such that $\mu-\beta$ is $\Delta^+(\frk,\frt)$-dominant, we have
	\begin{itemize}
		\item[(a)] If $b_1\ge 2p+1$, then $\Vert\mu\Vert_{\mathrm{spin}}>\Vert\mu-\beta\Vert_{\mathrm{spin}}$.
		\item[(b)] If $a_1\ge 2q+1$, then $\Vert\mu\Vert_{\mathrm{spin}}>\Vert\mu-\beta\Vert_{\mathrm{spin}}$.
	\end{itemize}
\end{prop}

\begin{proof}
	We only prove (a). If $a_1\ge q+1$, we are done. Now assume $a_1\le q$. Take any $\mu$-deficient $\rho_n^{(\ell)}=(K_\ell\mid R_\ell)=(k_1,\dots,k_p\mid r_1,\dots,r_q)$ and write
	\[\mu_\ell=(m_1, \dots, m_p\mid n_1, \dots, n_q).\]
Note that $k_1\ge a_1$ and \eqref{eq7} holds, namely,
	\[(k_1+r_1)-(a_1+b_1)+(p-q)+(N_\ell-M_\ell^+)+1\ge 0.\]
	According to the discussion at the end of Section 2, we have $k_1+r_1\le p+q-1$ and $N_{\ell}\leq k_1$. Therefore,
\begin{equation}\label{mu-deficient-necessary-condition}
k_1 \ge a_1+1, \quad N_{\ell} \ge a_1+1.
\end{equation}
Indeed, one deduces that
\[k_1\ge N_\ell-M_\ell^+\ge (a_1+b_1)-(k_1+r_1)+(q-p)-1\ge a_1+1.\]

We \emph{claim} that the spin norm of $\mu$ can not be attained at $\rho_n^{(\ell)}$. Firstly, consider the case $b_q\ge p$.
Indeed, let $h$ be the largest index for which the $h$-th coordinate of $K_\ell$ is at least $a_1$. We first handle the case $h>1$. Then \eqref{eq5} shows that
	\begin{equation}\label{eq-rhoL}
	\rho_n^{(\ell)}-\rho_n^{(L_{\ell, h})}=(\underbrace{0,\dots,0,1}_h,0,\dots,0\mid \underbrace{0,\dots,0,-1}_{q-k_h+1},0,\dots,0).
	\end{equation}
	
	Since $a_i\le a_1-1$ for $2\leq i\leq p$, then $m_h$ must be a negative integer. Notice that $r_j=p$ holds only for $0\le j\le q-k_1$, then $b_q\ge p$ and $q-k_h+1>q-k_1$ implies that $n_{q-k_h+1}$ is a positive integer. Lemma \ref{lemma-dominat} shows that
$$
\{\mu_\ell\}\gg\{\mu_\ell+\rho_n^{(\ell)}-\rho_n^{(L_{\ell, h})}\}=\{\mu_{L_{\ell, h}}\}.
$$
Then $\Vert\mu_\ell\Vert_\frk>\Vert \mu_{L_{\ell, h}}\Vert_\frk$. We call the movement from $\mu_\ell$ to $\mu_{L_{\ell, h}}$ an \emph{operation}.
Now if the $h$-th component of $K_{L_{\ell, h}}$ is still no less than $a_1$, we further \emph{operate} on $K_{L_{\ell, h}}$, and continue the process until the $h$-th component becomes $a_1-1$. At this stage, we are facing with certain $\rho_n^{(\ell')}\in \Omega_{p, q}$ such that  $h-1$ is the largest index of a coordinate of $K_{\ell'}$ that is at least $a_1$, and we continue in this manner until $h=1$.

In the case $h=1$, we carry out the operation whenever $k_1>a_1$. We end up with a $\rho_n^{(\ell\downarrow)}=(K_{\ell\downarrow}\mid R_{\ell\downarrow})\in \Omega_{p, q}$ such that the first component of $K_{\ell\downarrow}$ is $a_1$ and that $\{\mu_\ell\}\gg\{\mu_{\ell\downarrow}\}$.
Note that by \eqref{mu-deficient-necessary-condition}, $\rho_n^{(\ell\downarrow)}$ is \emph{not} $\mu$-deficient.
Thus the claim holds in the case $b_q\ge p$.
	
	More generally, consider the case that $b_{k}\ge p$ and $b_{k+1}<p$. Here $1\leq k<q$. Then
	\begin{equation}\label{ineq2}
	n_1\ge b_1-p\ge p+1>|n_{k+1}|\ge\cdots \ge |n_q|.
	\end{equation}
Note that $k+k_1\le q$ can not happen: otherwise we would have $N_{\ell}=0>a_1+1$, this contradicts the assumption that $\mu-\beta$ is $\Delta^+(\frk, \frt)$-dominant. Therefore,  $k+k_1>q$, which implies that $(q-k_1)+(q-k)<q$. Notice that
\[(\underbrace{|n_1|,\dots,|n_{q-k_1}|}_{q-k_1},\dots,|n_{k}|,\underbrace{|n_{k+1}|,\dots,|n_{q}|}_{q-k}),\]
which together with \eqref{ineq1} and \eqref{ineq2} implies that $N_{\ell}\le k_1-(q-k)$. Therefore, by \eqref{mu-deficient-necessary-condition},
\begin{equation}\label{ineq3}
k_1\ge (q-k)+N_{\ell} \ge a_1+(q-k)+1.
\end{equation}
Note that $a_1+(q-k)\ge q$
can not happen: otherwise \eqref{ineq3} would imply that $k_1\ge q+1$ which is ridiculous. Therefore, $a_1+(q-k)<q$. If there exists an index $h$ such that $k_h\ge a_1+(q-k)$, then
	\[q-k_h+1\le q-[a_1+(q-k)]+1=k-a_1+1\le k.\]
	Thus  $b_{q-k_h+1}\ge b_k\ge p$. Now let $h$ be the largest index for which the $h$-th coordinate of $K_\ell$ is at least $a_1+(q-k)$. The above discussion shows that $m_h$ is a negative integer and $n_{q-k_h+1}$ is a positive integer. Then \eqref{eq-rhoL} and Lemma \ref{lemma-dominat} implies $\Vert\mu_\ell\Vert_\frk>\Vert \mu_{L_{\ell, h}}\Vert_\frk$. This shows that the above operation and its iteration work as in the previous case. That is, there exists a $\rho_n^{(\ell\downarrow)}=(K_{\ell\downarrow}\mid R_{\ell\downarrow})\in \Omega_{p, q}$ such that the first coordinate of $K_{\ell\downarrow}$ is $a_1+(q-k)$ and $\{\mu_\ell\}\gg\{\mu_{\ell\downarrow}\}$. We emphasize that $\rho_n^{(\ell\downarrow)}$ is not $\mu$-deficient by \eqref{ineq3}. Thus the claim holds in this general situation.

Thus the claim always holds. Now if \[\Vert\mu\Vert_{\mathrm{spin}}=\Vert\mu-\rho_n^{(i_0)}\Vert_{\mathfrak{k}},\] then $\rho_n^{(i_0)}$ is not $\mu$-deficient by the claim. That is,
	\[\Vert\mu\Vert_{\mathrm{spin}}=\Vert\mu_{i_0}\Vert_\frk>\Vert\mu_{i_0}-\beta\Vert_\frk\ge \Vert\mu-\beta\Vert_{\mathrm{spin}}.\]
\end{proof}


In view of Proposition \ref{prop-2p+1}, we are left with the following cases:
\begin{itemize}
	\item[($\partial L$)] $2q\ge a_1\ge q+1$, $p\ge b_1\ge 1$, $a_1+b_1\le p+2q-1$.
	\item[($\partial R$)] $q\ge a_1\ge 1$, $2p\ge b_1\ge p+1$, $a_1+b_1\le 2p+q-1$.
\end{itemize}

Since $\mu$ is assumed to be u-large, by Proposition \ref{prop-u-large}, we have
$$
\mu\in\bigcup_{0\leq f\leq p, \, 0\leq g\leq q} \Lambda_{f, g}.
$$
If $f=0$, we have
\[\sum_{j=1}^g b_j>2pg,\]
and hence one must have $b_1\ge 2p+1$. Similarly, if $g=0$, we have
\[\sum_{i=1}^f a_i>2qf,\]
and therefore one must have $a_1\ge 2q+1$. Thus by Proposition \ref{prop-2p+1}, it remains to consider
$$
\mu\in\bigcup_{1\leq f\leq p, \, 1\leq g\leq q} \Lambda_{f, g}.
$$

The technique used in the Proposition 3.1 extends to the following lemma.
\begin{lemma}\label{lemma-down}
Take $\mu\in \Lambda_{f, g}$ for certain $1\leq f\leq p$ and $1\leq g\leq q$.
Let $\rho_n^{(\ell)}=(k_1,\dots,k_p\mid r_1,\dots,r_q)\in\Omega_{p,q}$. Assume that $k_1>a_1$.
	\begin{itemize}
		\item[(a)] If $\mu$ is a $\partial R$-weight, then exists  $\rho_n^{(\ell\downarrow)}=(K_{\ell\downarrow}\mid R_{\ell\downarrow})\in\Omega_{p, q}$ such that $\{\mu_\ell\}\gg\{\mu_{\ell\downarrow}\}$ and that
		\[K_{\ell\downarrow}=(a_1, \underbrace{a_1-1, \dots, a_1-1}_{h-1}, k_{h+1}, \dots, k_p),\]
		where $h$ is the largest index satisfying $k_h\ge a_1$.
		\item[(b)] If $\mu$ is a $\partial L$-weight, then exists $\rho_n^{(\ell\downarrow)}=(K_{\ell\downarrow}\mid R_{\ell\downarrow})\in\Omega_{p, q}$ such that $\{\mu_\ell\}\gg\{\mu_{\ell\downarrow}\}$ and that
		\[R_{\ell\downarrow}=(b_1, \underbrace{b_1-1, \dots, b_1-1}_{h-1}, r_{h+1}, \dots, r_q),\]
		where $h$ is the largest index satisfying $r_h\ge b_1$.
	\end{itemize}
\end{lemma}
\begin{proof}
	We  only prove (a).  Note that $a_1=1$ can not happen. Indeed, we would then have
	\[2pg+2f(q-g)=2pq-2(q-f)(p-g)<\sum_{i=1}^f a_i+\sum_{j=1}^g b_j\le 2+(2p-1)g.\]
  It follows that $g+2f(q-g)<2$. This inequality holds if and only if $g\le 1$ and $f=0$, contradicting the assumption. Therefore, $a_1\ge 2$.

Firstly, take $\mu\in\Lambda_{p, q}$. Then
	\[\sum_{i=1}^pa_i+\sum_{j=1}^qb_j>2pq.\]
	Let $a$ be the minimum of the first components among all $\partial R$-weights in $\Lambda_{p,q}$, then $a\ge 2$. Denote
	\begin{align*}
		\nu&=(a,a-1,\dots,a-1\mid 2p,2p-1,\dots,2p-1), \\
		\nu^*&=(a-1,a-2,\dots,a-2\mid 2p,2p-1,\dots,2p-1).
	\end{align*}
	It is clear that $\nu$ is in $\Lambda_{p,q}$ and $\nu^*$ is not in $\Lambda_{p,q}$. Moreover, $S(\nu)-S(\nu^*)=p$.
Take any $\partial R$-weight $\kappa$ in $\Lambda_{p,q}$ whose first component is $a$. Then
$$
S(\nu^*)\le 2pq< S(\kappa)\leq S(\nu)
$$
If the last component of $\kappa$ is less than or equal to $p-1$, then
$$
S(\kappa)\leq S(\nu)-p=S(\nu)-\big(S(\nu)-S(\nu^*)\big)=S(\nu^*).
$$
This is a contradiction. Thus the last component of $\kappa$ is at least $p$. Hence every $\partial R$-weight in $\Lambda_{p,q}$ with first component $a$ has all of its last $q$ components at least $p$. Furthermore, denote
	\[\widetilde{\nu}=(a+1,a,\dots,a\mid 2p,2p-1,\dots,2p-1).\]
	Then $S(\widetilde{\nu})-S(\nu)=p$.
Take any $\partial R$-weight $\widetilde{\kappa}$ in $\Lambda_{p,q}$ with first component $a+1$. Then
$$
S(\nu^*)<S(\widetilde{\kappa})\leq S(\widetilde{\nu}).
 $$
If the $(n-1)$-th component of $\kappa$ is less than or equal to $p-1$, then
$$
S(\widetilde{\kappa})\leq S(\widetilde{\nu})-2p=S(\nu)-p\leq S(\nu^*).
$$
This is a contradiction. Hence the $(n-1)$-th component of any $\partial R$-weight in $\Lambda_{p,q}$ with first component $a+1$ is at least $p$. This procedure continues until $a$ reaches $a_1$,  which is the first component of $\mu$. When the procedure terminates, we see that $b_{q-(a_1-a)}$ is not less than $p$. Thus $b_{q-a_1+1}\ge b_{q-(a_1-a)}\ge p$.
	
	Since \eqref{eq5} shows that
	\[\rho_n^{(\ell)}-\rho_n^{(L_{\ell, h})}=(\underbrace{0,\dots,0,1}_{h},0,\dots,0\mid 0,\dots,0,\underbrace{-1,0,\dots,0}_{k_{h}}).\]
	Note that the $h$-th component $a_h-k_h$ of $\mu_\ell$ is a negative integer. Moreover, $k_h>0$ implies that the $(n-k_h+1)$-th component of $\rho_n^{(\ell)}$ is less than $p$. Hence, by the above argument and $n-k_h+1\le n-a_1+1$, the $(n-k_h+1)$-th component of $\mu$ is at least $p$. Thus the $(n-k_h+1)$-th component of $\mu_\ell$ is a positive integer. Now, the same argument as in Proposition \ref{prop-2p+1} shows that after a finite number of operations, we obtain the desired result.
	
	Now we turn to the general case. That is, $\mu\in \Lambda_{f, g}$. We \emph{claim} that $q-a_1+1\le g$. It is obviously true if $g=q$. Now assume that $g<q$.
	Indeed, if $g \le q/2$, then \eqref{eq4} shows that
	\[\sum_{i=1}^f a_i+\sum_{j=1}^g b_j>2pq-2(p-f)(q-g) = 2pg+2f(q-g)\ge 2pg+fq.\]
On the other hand,
$$
\sum_{i=1}^f a_i+\sum_{j=1}^g b_j\leq fq+2pg
$$
since each $a_i$ is at most $q$ and each $b_j$ is at most $2p$. This is a contradiction. Therefore, $g>q/2$.
	
	Assume that $3q/4\ge g>q/2$. Then $a_1\le q/2$ can not happen. Indeed, we would then have
	\[\dfrac{1}{2}fq+2pg \ge \sum_{i=1}^f a_i+\sum_{j=1}^g b_j >2pg+2f(q-g).\]
	This implies that $g>\dfrac{3}{4}q$, which is a contradiction. Hence $a_1>q/2$. Then
	\[q-a_1+1\leq\left[\dfrac{1}{2}q \right]+1\le g.\]
	Here for any $x\in\bbR$, $[x]$ denotes the smallest integer which is less than or equal to $x$. Thus the claim holds in this case.
	
	More generally, assume that $(1-1/2^{d+1})q\ge g>(1-1/2^d)q$. Then $a_1\le q/2^d$ can not happen. Indeed, we would then have
	\[\dfrac{1}{2^{d}}fq+2pg\ge \sum_{i=1}^f a_i+\sum_{j=1}^g b_j>2pg+2f(q-g).\]
	This implies that
	$g>(1-1/2^{d+1})q$, which is a contradiction.
	Therefore, $a_1>q/2^{d}$, then
	\[q-a_1+1\le \left[(1-\dfrac{1}{2^d})q \right]+1\le g.\]
	Thus the claim holds in this case.

Since $q$ is a fixed constant, there exists a positive integer $D$ such that $2^D\ge q$. Then
$$
(1-1/2^D)q\ge q-1.
$$
Therefore, the above procedure terminates when $d=D$ and the claim follows. It follows from the claim that $b_{q-a_1+1}\ge b_g$.
	
	Now let $a\ge 2$ be the minimum of the first components of all $\partial R$-weights in $\Lambda_{f,g}$. Put
	\[\nu=(\underbrace{a,a-1,\dots,a-1}_{f},a-2,a-2,\dots,a-2 \mid 2p,2p-1,\dots,2p-1),\]
	then $\nu$ is a $\partial R$-weight in $\Lambda_{f,g}$. Thus $S(\nu_{f, g})>2pq$ by \eqref{sum-mufg}. Construct
	\[\nu^*=(\underbrace{a-1,a-2,\dots,a-2}_{f},a-2,a-2,\dots,a-2 \mid 2p,2p-1,\dots,2p-1),\]
	then $\nu^*$ not in $\Lambda_{f,g}$. Thus $S(\nu_{f, g}^*)\leq 2pq$ by \eqref{sum-mufg}. Note that $S(\nu_{f,g})-S(\nu_{f,g}^*) = f\le p$. Take any $\partial R$-weight $\kappa$ in $\Lambda_{f,g}$ whose first component is $a$. Then
	\[S(\nu^*_{f,g}) < S(\kappa_{f,g}) \le S(\nu_{f,g}).\]
If the $(p+g)$-th component of $\kappa$ is less than or equal to $p-1$, then
\[S(\kappa_{f, g})\leq S(\nu_{f,g})-p\leq S(\nu_{f,g})-f=S(\nu_{f,g})-\big(S(\nu_{f,g})-S(\nu_{f,g}^*)\big)=S(\nu_{f,g}^*).\]
This is a contradiction. Hence the $(p+g)$-th component of any $\partial R$-weight in $\Lambda_{f,g}$ with first component $a$ is at least $p$. In other words, $b_1\ge \dots \ge b_g\ge p$.  Notice that
	\[a_1-a\le q-a< g,\]
where the last inequality follows from the claim.
Proceeding as in the case of $\mu\in \Lambda_{p, q}$, we obtian that $b_{g-(a_1-a)}\ge p$. Thus $b_{g-a_1+1}\ge b_{g-(a_1-a)}\ge p$. Therefore the discussion in the $\Lambda_{p, q}$ case applies here.
\end{proof}

\subsection{The boundary case}

Theorem 1.1 will be proved by means of the following boundary descent proposition.

\begin{prop}\label{prop-boundary}
	Assume that $\mu$ is a $\partial R$-weight or a $\partial L$-weight. Let $\rho_n^{(\ell)}=(k_1, \dots, k_p \mid r_1, \dots, r_q)\in \Omega_{p,q}$ be $\mu$-deficient. Then $\Vert\mu_\ell\Vert_\frk>\Vert\mu\Vert_{\mathrm{spin}}$.
\end{prop}
\begin{proof}
	We only treat the case that $\mu$ is a $\partial R$-weight. Assume $\mu \in \Lambda_{f,g}$, then
\begin{equation}\label{eq9}
\sum_{i=1}^f a_i+\sum_{j=1}^g b_j>2pq-2(p-f)(q-g).
\end{equation}
	By Lemma \ref{lemma-down}, without loss of generality, we assume that $k_1=a_1$ and $k_2\le a_1-1$. Put
	\[A_i=q-a_i, 1\leq i\leq p; \quad c_j=\card\{i\mid A_i\ge j\}, 1\leq j\leq q.\]
	We claim that
	\[\rho_n^{(\ell^*)}:=(a_1,\dots,a_p\mid c_1,\dots,c_q)\in\Omega_{p,q}.\]
	Indeed, it is obvious that $q\ge a_1\ge\cdots\ge a_p\ge 0$ and $p\ge c_1\ge\cdots\ge c_q\ge 0$. For each fixed $i$, it contributes exactly to $c_1,\dots,c_{q-a_i}$, so the total contribution is $q-a_i$. Hence
	\[\sum_{j=1}^qc_j=\sum_{i=1}^p(q-a_i)=pq-\sum_{i=1}^pa_i.\]
	This implies $\sum_{i=1}^pa_i+\sum_{j=1}^qc_j=pq$. Therefore, $\rho_n^{(\ell^*)}\in\Omega_{p,q}$. Lemma \ref{lemma-omegapq} then yields
	\begin{equation}\label{eq8}
	\sum_{i=1}^f a_i+\sum_{j=1}^g c_j\le fq+(p-f)g.
	\end{equation}
	We \emph{claim} that there exists $1\leq t\le q$ such that
	\[b_t\ge p+c_t+1.\]
	Suppose otherwise. Then $b_j\le p+c_j$ holds for every $1\le j\le q$. This implies
	\[\sum_{j=1}^g b_j\le pg+\sum_{j=1}^g c_j.\]
	Substituting this into \eqref{eq8}, we obtain
	\[\sum_{i=1}^f a_i+\sum_{j=1}^g b_j\le fq+(2p-f)g.\]
	On the other hand,
	\[[2pq-2(p-f)(q-g)]-[fq+(2p-f)g]=f(q-g)\ge 0.\]
	Hence
	\[\sum_{i=1}^f a_i+\sum_{j=1}^g b_j\le 2pq-2(p-f)(q-g).\]
	This contradicts \eqref{eq9}. Therefore, the claim holds. That is, there exists $1\leq t_0\le q$ such that $b_{t_0}\ge p+c_{t_0}+1$.
	
	Write $\mu-\rho_n^{(\ell)}=(m_1,\dots,m_p\mid n_1,\dots,n_q)$, and set
	\[G_i=q-k_i, 1\leq i\leq p; \quad Z=\card\{i\mid m_i=0\}.\]
	It is immediate that $m_i=G_i-A_i$. Define
	\[S_{t_0}=\{i\mid m_i=0,A_i<t_0\}.\]
	Recall that $c_{t_0}=\card\{i\mid A_i\ge t_0\}$. Hence there are at most $c_{t_0}$ indices satisfying $A_i\ge t_0$ and $m_i=0$. Thus there are at least $Z-c_{t_0}$ indices $i\in\{1, 2, \dots, p\}$ satisfying $A_i<t_0$ and $m_i=0$. That is,
	\[\card(S_{t_0})\ge Z-c_{t_0}.\]
	Let us show that the set $S_{t_0}$ is nonempty. Since $A_1\le \cdots\le A_p$, if $A_1\ge t_0$, then $c_{t_0}=p$. In this case, $b_{t_0}\ge p+c_{t_0}+1=2p+1$, which contradicting the fact that $\mu$ is a $\partial R$-weight. Therefore, $A_1<t_0$ must hold. Recall that we have assumed that $a_1=k_1$. Thus $m_1=0$. Therefore, $1\in S_{t_0}$.
	
	Denote $j_0:=\max S_{t_0}$,  $g_0:=G_{j_0}=A_{j_0}$. Set
	\[h=\max\,\{i\mid G_i=g_0\}.\]
	Since $G_1\le \cdots\le G_p$, we have
	\[h=p\quad\mbox{or}\quad G_h<G_{h+1}.\]
	 When $h=p$, we have $G_p=g_0=A_{j_0}<t_0\le q$. Thus $G_p<q$. Therefore, $k_p>0=k_{p+1}$. On the one hand, $G_h<G_{h+1}$ is equivalent to $k_h>k_{h+1}$.
	
	Next, we \emph{claim} that $m_h\le 0$. Indeed, since $j_0 \le h$, we have
	\[A_h\ge A_{j_0}=g_0=G_h.\]
Thus
	\[m_h=G_h-A_h\le 0.\]
	
	Since $g_0=A_{j_0}<t_0$, we have that $t_h:=q-k_h+1=G_h+1=g_0+1\le t_0$. Since $k_h > k_{h+1}$, Section \ref{subsec-rhon} shows that $r_{t_h}=p-h$. Because $j_0=\max S_{t_0}$, we have
$$
j_0\ge\card(S_{t_0})\ge Z-c_{t_0}.$$
Hence
	\[h\ge j_0\ge Z-c_{t_0}.\]
	Finally,
	\[n_{t_h}=b_{t_h}-r_{t_h}\ge b_{t_0}-(p-h)\ge (p+c_{t_0}+1)-(p-h)=h+c_{t_0}+1 \ge Z+1.\]
	
	Now let $\rho_n^{(\ell\downarrow)}=\rho_n^{(L_{\ell,h})}$. Then \eqref{eq5} shows that
	\[\rho_n^{(\ell)}-\rho_n^{(\ell\downarrow)}=(\underbrace{0,\dots,0,1}_{h},0,\dots,0 \mid \underbrace{0,\dots,0,-1}_{t_h},0,\dots,0).\]
	If $m_h<0$, then Lemma \ref{lemma-dominat} implies $\{\mu_\ell\}\gg\{\mu_\ell+(\rho_n^{(\ell)}-\rho_n^{(\ell\downarrow)})\}=\{\mu_{\ell\downarrow}\}$, and consequently $\Vert\mu-\rho_n^{(\ell\downarrow)}\Vert_\frk<\Vert\mu-\rho_n^{(\ell)}\Vert_\frk$. If $m_h=0$, then
	\[\Vert\mu_{\ell}\Vert^2_\frk-\Vert\mu_{\ell\downarrow}\Vert^2_\frk\ge[Z^2-(Z+1)^2]+[(n_{t_h}+1)^2-n_{t_h}^2]=2(n_{t_h}-Z)\ge 2.\]
So we  still have
	\[\Vert\mu\Vert_{\mathrm{spin}}\le\Vert\mu_{\ell\downarrow}\Vert_\frk <\Vert\mu_\ell\Vert_\frk.\]
\end{proof}


\begin{proof}[Proof of Theorem \ref{thm-reduction-pencil-sppq}]
As mentioned after Proposition \ref{prop-2p+1},	it suffices to consider the case that $\mu$ is a $\partial R$-weight or a $\partial L$-weight. In this case, let $\Vert\mu\Vert_{\mathrm{spin}}=\Vert\mu-\rho_n^{(i_0)}\Vert_\frk$. By Proposition \ref{prop-boundary}, $\rho_n^{(i_0)}$ is not $\mu$-deficient. Then
	\[\Vert\mu\Vert_{\mathrm{spin}}=\Vert\mu-\rho_n^{(i_0)}\Vert_\frk>\Vert\mu-\rho_n^{(i_0)}-\beta\Vert_\frk\ge\Vert\mu-\beta\Vert_{\mathrm{spin}}.\]
\end{proof}

\medskip


\end{document}